\documentclass[12pt]{amsart}

\usepackage{amsfonts,amssymb,amsmath,amsthm}
\usepackage{enumerate}

\usepackage{graphicx}
\usepackage{mathtools}
\usepackage{amssymb}
\usepackage{amsmath, amsthm}
\usepackage{amscd}
\usepackage{enumitem}
\usepackage[all]{xy}
\usepackage[margin=2.5cm]{geometry}

\newenvironment{axiom}[1]
{
\par\smallskip\noindent
{\bf Axiom #1}\begin{it}
}
{
\end{it}\par\smallskip
 }

\swapnumbers
\newtheorem{theorem}[subsubsection]{Theorem}
\newtheorem{lemma}[subsubsection]{Lemma}

\newtheorem{corollary}[subsubsection]{Corollary}

\newtheorem{remark}[subsubsection]{Remark}

\newtheorem{definition}[subsubsection]{Definition}

\makeatletter
\@addtoreset{equation}{section}
\@addtoreset{figure}{section}
\@addtoreset{table}{section}
\makeatother

\makeatletter
\newcommand\testshape{family=\f@family; series=\f@series; shape=\f@shape.}
\def\myemphInternal#1{\if n\f@shape%
\begingroup\itshape #1\endgroup\/%
\else\begingroup\bfseries #1\endgroup%
\fi}
\def\myemph{\futurelet\testchar\MaybeOptArgmyemph}
\def\MaybeOptArgmyemph{\ifx[\testchar \let\next\OptArgmyemph
                 \else \let\next\NoOptArgmyemph \fi \next}
\def\OptArgmyemph[#1]#2{\index{#1}\myemphInternal{#2}}
\def\NoOptArgmyemph#1{\myemphInternal{#1}}
\makeatother

\newcommand\CCC{{\mathbb C}}

\newcommand\NNN{{\mathbb N}}

\newcommand\RRR{{\mathbb R}}

\newcommand\ZZZ{{\mathbb Z}}

\newcommand\FF{{\mathcal F}}

\newcommand\id{\mathrm{id}}          
\newcommand\Per{\mathrm{Per}}        

\newcommand\AxCrPt{{\text{\rm(L)}}}
\newcommand\AxBd{{\text{\rm(B)}}}

\newcommand\Mman{M}

\newcommand\Pman{P}
\newcommand\Qman{Q}

\newcommand\Uman{U}

\newcommand\Orb{\mathcal{O}}        
\newcommand\Stab{\mathrm{Stab}}       
\newcommand\Diff{\mathrm{Diff}}       

\newcommand\Vol{\mathrm{Vol}}       
\newcommand\Symp{\mathrm{Symp}}     

\newcommand\DiffId{\Diff_{0}}     
\newcommand\StabId{\Stab_{0}}     

\newcommand\Cinfty{\mathcal{C}^{\infty}}

\newcommand\func{f}
\newcommand\gfunc{g}
\newcommand\dif{h}

\newcommand\DiffM{\Diff(\Mman)}

\newcommand\DiffIdM{\DiffId(\Mman)}

\newcommand\Morse{\mathrm{Morse}}

\newcommand\Stabilizer[1]{\Stab(#1)}             
\newcommand\StabilizerPlus[1]{\Stab^{+}(#1)}     
\newcommand\StabilizerId[1]{\StabId(#1)}         

\newcommand\Orbit[1]{\Orb(#1)}                   

\newcommand\SingularSet[1]{\Sigma_{#1}}             

\newcommand\AutKRGraphStab[1]{\mathbf{G}}             

\newcommand\fSing{\SingularSet{\func}}                

\newcommand\fld{F}
\newcommand\flow{\mathbf{F}}

\newcommand\image{\mathrm{image}}

\newcommand\Rhalf[1]{\RRR_{+}^{#1}}

\newcommand\hFld{H}
\newcommand\hFlow{\mathbf{\hFld}}

\newcommand\sh{\varphi}

\newcommand\ZentrQ[2]{\mathcal{Z}_{\omega}^{#1}(#2)}

\newcommand\KRGraph{K}

\newcommand\df[2]{#1(#2)}

\newcommand\VolM{\Vol(\Mman,1)}
\begin{document}

\author{Sergiy Maksymenko}
\title{Symplectomorphisms of surfaces preserving a smooth function, I}
\address{Institute of Mathematics of NAS of Ukraine, Te\-re\-shchen\-kivska st. 3, Kyiv, 01004 Ukraine} 
\email{maks@imath.kiev.ua}
\urladdr{http://www.imath.kiev.ua/~maks}
\thanks{The author is indebted to Bogdan Feshchenko for useful discussions.}



\begin{abstract}
Let $M$ be a compact orientable surface equipped with a volume form $\omega$, $P$ be either $\mathbb{R}$ or $S^1$, $f:M\to P$ be a $C^{\infty}$ Morse map, and $H$ be the Hamiltonian vector field of $f$ with respect to $\omega$.
Let also $\mathcal{Z}_{\omega}(f) \subset C^{\infty}(M,\mathbb{R})$ be set of all functions taking constant values along orbits of $H$, and $\mathcal{S}_{\mathrm{id}}(f,\omega)$ be the identity path component of the group of diffeomorphisms of $M$ mutually preserving $\omega$ and $f$.

We construct a canonical map $\varphi: \mathcal{Z}_{\omega}(f) \to \mathcal{S}_{\mathrm{id}}(f,\omega)$ being a homeomorphism whenever $f$ has at least one saddle point, and an infinite cyclic covering otherwise.
In particular, we obtain that $\mathcal{S}_{\mathrm{id}}(f,\omega)$ is either contractible or homotopy equivalent to the circle.

Similar results hold in fact for a larger class of maps $\Mman\to\Pman$ whose singularities are equivalent to homogeneous polynomials without multiple factors.
\end{abstract}

\maketitle

\section{Introduction}
Let $\Mman$ be a closed oriented surface, $\DiffM$ be the group of all $C^{\infty}$ diffeomorphisms of $\Mman$, and $\DiffIdM$ be the identity path component of $\DiffM$ consisting of all diffeomorphisms isotopic to the identity.

Let also $\VolM$ be the space of all volume forms on $\Mman$ having volume $1$ and $\omega\in\VolM$.
Since $\dim\Mman=2$, $\omega$ is a closed non-degenerate $2$-form and so it defines a symplectic structure on $\Mman$.
Denote by $\Symp(\Mman,\omega)$ the group of all $\omega$-preserving $C^{\infty}$ diffeomorphisms, and let $\Symp_0(\Mman,\omega)$ be its identity path component.

Then Moser's stability theorem~\cite{Moser:TrAMS:1965} implies that for any $C^{\infty}$ family 
\[ \{\omega_t\}_{t\in D^n} \subset \VolM\]
of volume forms parameterized by points of a closed $n$-dimensional disk $D^n$, there exists a $C^{\infty}$ family of diffeomorphisms 
\[\{\dif_t\}_{t\in D^n} \subset \DiffIdM\]
such that $\omega_t = \dif_t^{*} \omega$ for all $t\in D^n$.
In particular, this implies that the map 
\begin{align*}
&p:\DiffIdM \to \VolM, & p(\dif) &= \dif^{*}\omega
\end{align*}
is a Serre fibration with fiber $\Symp_0(\Mman,\omega)$, see e.g.\! \cite[\S3.2]{McDuffSalamon:SympTop:1995}, \cite{Banyaga:CMH:1978}, or~\cite[\S7.2]{Polterovich:SympDiff:2001}.

Since $\VolM$ is convex and therefore contractible, it follows from exact sequence of homotopy groups of the Serre fibration $p$ that $p$ yields isomorphisms of the corresponding homotopy groups $\pi_k\Symp_0(\Mman,\omega) \cong \pi_k\DiffIdM$, $k\geq0$.
Hence the inclusion
\begin{equation}\label{equ:Symp0_diffIdM}
\Symp_0(\Mman,\omega) \ \subset \ \DiffIdM
\end{equation}
turns out to be a weak homotopy equivalence.
See also \cite{McDuff:ProcAMS:1985} for discussions of the inclusion~\eqref{equ:Symp0_diffIdM} for non-compact manifolds.

Moreover, let $\Diff^{+}(\Mman)$ be the group of orientation preserving diffeomorphisms.
Then we have an inclusion $i:\Symp(\Mman,\omega) \subset \Diff^{+}(\Mman)$.
Indeed, if $\dif$ preserves $\omega$, then it fixes the corresponding cohomology class $[\omega]\in H^2(\Mman,\RRR) \cong \RRR$, and so yields the identity on $H^2(\Mman,\RRR)$.
In particular, $\dif$ preserves orientation of $\Mman$.
Hence~\eqref{equ:Symp0_diffIdM} also implies that $i$ yields a monomorphism $i_0:\pi_0\Symp(\Mman,\omega) \to \pi_0\Diff^{+}(\Mman)$ on the set of isotopy classes.

It is well known that $\pi_0\Diff^{+}(\Mman)$ is generated by isotopy classes of Dehn twists,~\cite{Dehn:AM:1938}, \cite{Lickorish:PCPS:1964}, and one easily shows that each Dehn twist can be realized by $\omega$-preserving diffeomorphism.
This implies that $i_0$ is also surjective, and so $i$ is a weak homotopy equivalence as well.

On the other hand, let $\func:\Mman\to\RRR$ be a Morse function,
\[
\Stabilizer{\func} = \{\dif\in\DiffM \mid \func\circ\dif = \func \}
\]
be the group of $\func$-preserving diffeomorphisms, i.e. the \myemph{stabilizer} of $\func$ with respect to the right action of $\DiffM$ on $C^{\infty}(\Mman,\RRR)$, and $\StabilizerId{\func}$ be its identity path component.
Let also 
\[
\Orbit{\func} = \{ \func\circ\dif \mid \dif\in\DiffM \}
\]
be the corresponding \myemph{orbit} of $\func$,
\[
\Stabilizer{\func,\omega} = \Stabilizer{\func} \cap \Symp(\Mman,\omega)
\]
be the group of diffeomorphisms mutually preserving $\func$ and $\omega$, and $\StabilizerId{\func,\omega}$ be its identity path component.

In a series of papers the author proved that $\StabilizerId{\func}$ is either \myemph{contractible} or \myemph{homotopy equivalent to the circle} and computed the higher homotopy groups of $\Orbit{\func}$, \cite{Maksymenko:AGAG:2006}, \cite{Maksymenko:ProcIM:ENG:2010}; showed that $\Orbit{\func}$ is homotopy equivalent to a finite-dimensional CW-complex,~\cite{Maksymenko:TrMath:2008}; and recently described precise algebraic structure of the fundamental group $\pi_1\Orbit{\func}$,~\cite{Maksymenko:DefFuncI:2014}.
E.~Kudryavtseva, \cite{Kudryavtseva:MathNotes:2012}, \cite{Kudryavtseva:MatSb:2013}, studied the homotopy type of the space of Morse maps on compact surfaces and using similar ideas as in~\cite{Maksymenko:AGAG:2006}, \cite{Maksymenko:ProcIM:ENG:2010} proved that $\Orbit{\func}$ has the homotopy type of a quotient of a torus by a free action of a certain finite group.

The present paper is former in a series subsequent ones devoted to extension of the above results to the case of $\omega$-preserving diffeomorphisms.
We will describe here the homotopy type of $\StabilizerId{\func,\omega}$.
In next papers will study the homotopy type of the subgroup of $\Stabilizer{\func,\omega}$ trivially acting on the Kronrod-Reeb graph of $\func$, see~\S\ref{sect:Kronrod_Reeb_graph}, and describe the precise algebraic structure of $\pi_0\Stabilizer{\func,\omega}$.

Notice that if $\hFld$ is the Hamiltonian vector field of $\func$ and $\hFlow:\Mman\times \RRR\to\Mman$ is the corresponding Hamiltonian flow, then $\hFlow_t\in\Stabilizer{\func,\omega}$ for all $t\in\RRR$.

More generally, given a $C^{\infty}$ function $\alpha:\Mman\to\RRR$, one can define the map 
\begin{align*}
&\hFlow_{\alpha}: \Mman\to\Mman, & \hFlow_{\alpha}(x) = \hFlow(x,\alpha(x)),
\end{align*}
being in general just a $C^{\infty}$ map leaving invariant each orbit of $\hFld$, and so preserving $\func$.
However, $\hFlow_{\alpha}$ is not necessarily a diffeomorphism.

Let $\mathcal{Z}(\func)=\{\alpha\in C^{\infty}(\Mman,\RRR) \mid \df{\hFld}{\alpha} = 0 \}$ be the algebra of all smooth functions taking constant values along orbits of $\hFlow$.
Equivalently, $\mathcal{Z}(\func)$ is the \myemph{centralizer} of $\func$ with respect to the Poisson bracket induced by $\omega$, see~\S\ref{sect:Poisson_mult}.
In Lemma~\ref{lm:Z_as_func_on_KR} we also identify $\mathcal{Z}(\func)$ with a certain subset of continuous functions on the Kronrod-Reeb graph of $\func$.
In particular, $\mathcal{Z}(\func)$ contains all constant functions.

We will prove in Theorem~\ref{th:charact_Sid_f_omega} that $\hFlow_{\alpha} \in \StabilizerId{\func,\omega}$ if and only if $\alpha\in\mathcal{Z}(\func)$.
Moreover if $\func$ has at least one saddle critical point, then the correspondence $\alpha\mapsto\hFlow_{\alpha}$ is a homeomorphism $\mathcal{Z}(\func) \cong \StabilizerId{\func,\omega}$ with respect to $C^{\infty}$ topologies, and so $\StabilizerId{\func,\omega}$ is contractible.
Otherwise, that correspondence is an infinite cyclic covering map and $\StabilizerId{\func,\omega}$ is homotopy equivalent to the circle.
It will also follow that the inclusion 
\[
 \StabilizerId{\func,\omega} \ \subset \ \StabilizerId{\func}
\]
is a homotopy equivalence.
This statement can be regarded as an analogue of~\eqref{equ:Symp0_diffIdM} for $\func$-preserving diffeomorphisms.

Again it implies that the inclusion \[j:\Stabilizer{\func,\omega} \ \subset \ \StabilizerPlus{\func} \equiv \Stabilizer{\func}\cap\Diff^{+}(\Mman)\] yields an injection $j_0:\pi_0\Stabilizer{\func,\omega} \to \pi_0\StabilizerPlus{\func}$ on the sets of isotopy classes.
However, now $j_0$ is not necessarily surjective, see~\S\ref{sect:j0_non_surj}.
The reason is that $\StabilizerPlus{\func}$ has many invariant subsets, e.g.\! the sets of the from $\Mman_a = \func^{-1}(-\infty,a]$, $a\in\RRR$, and so if $h\in\Stabilizer{\func,\omega}$ interchanges connected components of $\Mman_a$, then they must have the same $\omega$-volume.

In fact, our results hold for a larger class of smooth maps $\func$ from $\Mman$ into $\RRR$ and $S^1$, see~\S\ref{sect:class_F}.
On the other hand, we also provide in \S\ref{counterexample:Zg_Sg} an example of a function with isolated critical points for which the above correspondence $\alpha\mapsto\hFlow_{\alpha}$ is not surjective.

The author is indebted to Bogdan Feshchenko for useful discussions.

\section{Preliminaries}
\subsection{Shift map}
Let $\Mman$ be a connected $n$-dimensional $C^{\infty}$ manifold, $\hFld$ be a $C^{\infty}$ vector field tangent to $\partial\Mman$ and generating a flow $\hFlow:\Mman\times\RRR\to\Mman$.
For each $\alpha\in C^{\infty}(\Mman,\RRR)$ define the following $C^{\infty}$ map $\hFlow_{\alpha}:\Mman\to\Mman$ by
\[ 
\hFlow_{\alpha}(x) = \hFlow(x,\alpha(x)),
\]
for $x\in\Mman$.
Evidently, $\hFlow_{\alpha}$ leaves invariant each orbit of $\hFlow$ and is homotopic to $\id_{\Mman}$ by the homotopy $\{\hFlow_{t\alpha}\}_{t\in[0,1]}$.
Also notice that if $\alpha\equiv t$ is a constant function, then $\hFlow_{\alpha}=\hFlow_{t}$ is a diffeomorphism belonging to the flow $\hFlow$.

For $\alpha\in C^{\infty}(\Mman,\RRR)$ we will denote by $\df{\hFld}{\alpha}$ the Lie derivative of $\alpha$ along $\hFld$.

\begin{lemma}\label{lm:jacobi_flow}{\rm\cite[Theorem~19]{Maksymenko:TA:2003}}
Let $\alpha\in C^{\infty}(\Mman,\RRR)$, $y\in\Mman$, and $z=\hFlow_{\alpha}(y)$.
Then the tangent map $T_y \hFlow_{\alpha}: T_y\Mman \to T_z\Mman$ is an isomorphism if and only if $1+\df{\hFld}{\alpha}(y) \not=0$.
\end{lemma}
\begin{remark} \rm
In fact, \cite[Lemma~20]{Maksymenko:TA:2003}, if $\alpha(y)=0$, so $z = \hFlow_{\alpha}(y)=\hFlow(y,0)=y$ is a fixed point of $\hFlow_{\alpha}$, then the determinant of $T_y \hFlow_{\alpha}: T_y\Mman \to T_y\Mman$ does not depend on a particular choice of local coordinates at $z$ and equals $1+\df{\hFld}{\alpha}(y)$.
The general case $\alpha(y)=a\not=0$ reduces to $a=0$ by observation that $\hFlow_{\alpha} = \hFlow_{\alpha-a} \circ \hFlow_{a}$.
\end{remark}
To get a global variant of Lemma~\ref{lm:jacobi_flow} notice that the correspondence $\alpha\mapsto\hFlow_{\alpha}$ can also be regarded as the following mapping 
\begin{align*}
&\sh_{\hFld}: C^{\infty}(\Mman,\RRR) \to C^{\infty}(\Mman,\Mman), &
\sh_{\hFld}(\alpha) = \hFlow_{\alpha}.
\end{align*}
It will be called the \myemph{shift map} along orbits of $\hFlow$, \cite{Maksymenko:TA:2003}, \cite{Maksymenko:OsakaJM:2011}. 
Consider the following subset of $C^{\infty}(\Mman,\RRR)$:
\begin{equation}\label{equ:Gamma_set}
\Gamma_{\hFld} = \{\alpha\in C^{\infty}(\Mman,\RRR) \mid 1 + \df{\hFld}{\alpha}>0 \},
\end{equation}
and let $\DiffId(\hFld)$ be the group of all diffeomorphisms of $\Mman$ which leave invariant each orbit of $\hFld$ and isotopic to the identity via an orbit preserving isotopy.

\begin{lemma}\label{lm:shift_are_diffeo}{\rm\cite[Theorem~19]{Maksymenko:TA:2003}}
If $\Mman$ is compact, then 
\begin{align}\label{equ:shift_being_diffeos}
&\sh(\Gamma_{\hFld}) \subset \DiffId(\hFld), & 
\Gamma_{\hFld} = \sh^{-1}\bigl(\DiffId(\hFld)\bigr).
\end{align}
In other words, suppose $\alpha\in C^{\infty}(\Mman,\RRR)$.
Then $\alpha\in\Gamma$ if and only if $\hFlow_{\alpha} \in \DiffId(\hFld)$. 
\end{lemma}

\subsection{Hamiltonian vector field}
Let $\Mman$ be a compact orientable surface equipped with a volume form $\omega$ and $\Pman$ be either $\RRR$ or $S^1$.
Since $\dim\Mman = 2$, $\omega$ is a closed $2$-form, and therefore it defines a symplectic structure on $\Mman$.
Then for each $C^1$ map $\func:\Mman\to\Pman$ there exists a unique vector field $\hFld$ on $\Mman$ satisfying
\begin{equation}\label{equ:ham_v_f}
d\func(z)(u) = \omega(u,\hFld(z)),
\end{equation}
for each point $z\in\Mman$ and a tangent vector $u\in T_{z}\Mman$.
This vector field is called the \myemph{Hamiltonian} vector field of $\func$ with respect to $\omega$.
For the convenience of the reader we recall its construction as it is usually defined for functions $\func:\Mman\to\RRR$ only.

Let $z\in\Mman$.
Fix local charts $\dif:\Uman\to\Mman$ and $q:J\to\Pman$ at $z$ and $\func(z)$ respectively, where $\Uman$ is an open subset of the upper half-plane 
$\Rhalf{2}=\{(x,y) \mid y\geq0\}$ and $J$ is an open interval in $\RRR$.
Decreasing $\Uman$ one can assume that $\func(\dif(\Uman)) \subset q(J)$.
Then the map $\hat{\func}=q^{-1} \circ \func \circ \dif:\Uman \to J$ is called a \myemph{local representation} of $\func$ at $z$.

Now if in coordinates $(x,y)$ on $\Uman$ we have that $\omega(x,y) = \gamma(x,y) dx \wedge dy$ for some non-zero $C^{\infty}$ function $\gamma:\Uman\to\RRR\setminus\{0\}$, then 
\begin{equation}\label{equ:ham_v_f_formula}
\hFld(x,y) = \frac{1}{\gamma(x,y)}\bigl(- \hat{\func}'_{y} \tfrac{\partial}{\partial x} + \hat{\func}'_{x} \tfrac{\partial}{\partial y}\bigr).
\end{equation}

A definition of $\hFld$ that does not use local coordinates can be given as follows.
Since the restriction of $\omega$ to each tangent space $T_x\Mman$ is a non-degenerate skew-symmetric form, it follows that $\omega$ yields a bundle isomorphism
\[
\xymatrix{
T\Mman \ar[rd] \ar[rr]^-{\psi} & & T^{*}\Mman \ar[ld] \\
&\Mman
}
\]
defined by the formula $\psi(u)(v) = \omega(u, v)$ for all $u,v\in T_x\Mman$ and $x\in\Mman$.

Further notice, that the tangent bundle of $\Pman$ is trivial, so we have the \myemph{unit} section
\begin{align*}
 s&: \Pman \to T\Pman\equiv\Pman\times\RRR, & s(q) &= (q,1).
\end{align*}

Now for a $C^1$ map $f:\Mman\to\Pman$ its \myemph{differential} $d\func:\Mman\to T^{*}\Mman$ and the \myemph{Hamiltonian} vector field $\hFld:\Mman\to T\Mman$ are unique maps for which the following diagram is commutative:
\[
\xymatrix{
TM \ar[rr]^-{\psi} && T^{*}\Mman && T^{*}\Pman \equiv \Pman\times\RRR \ar[ll]_-{T^{*}f} \\
&& \Mman \ar[rr]^-{\func} \ar[ull]^-{\hFld} \ar[u]_-{d\func} && \Pman \ar[u]_-{s} 
}
\]
Thus $d\func = T^{*}\func \circ s \circ \func$, and $\hFld = \psi^{-1}\circ d\func$.
It follows that 
\begin{equation}\label{equ:df_F_0}
\df{\hFld(z)}{\func} = \omega(\hFld(z), \hFld(z)) = 0,
\end{equation}
as $\omega$ is skew-symmetric, and so \myemph{$\hFld$ is tangent to level curves of $\func$}.

Suppose, in addition, that $\func$ takes constant values at boundary components of $\Mman$.
Then, due to~\eqref{equ:df_F_0}, $\hFld$ is tangent to $\partial\Mman$, and therefore it yields a flow $\hFlow:\Mman\times\RRR\to\Mman$.
It also follows from~\eqref{equ:df_F_0} that each diffeomorphism $\hFlow_{t}:\Mman\to\Mman$ preserves $\func$, in the sense that $\func\circ\hFlow_{t} = \func$.
Moreover, the well known Liouville's theorem claims that each diffeomorphism $\hFlow_{t}$ also preserves $\omega$.
In fact, that theorem is a simple consequence of Cartan's identity:
\begin{align}\label{equ:Ht_prev_omega}
\mathcal {L}_{\hFld}\omega &=
   d( \iota_{\hFld}\omega ) + \iota_{\hFld} d\omega  = 
   d( d\func ) + \iota_{\hFld} 0 = 0,
\end{align}
since $\iota_{\hFld}\omega = \omega(\hFld, \cdot) = d\func$ by~\eqref{equ:ham_v_f}, and $d\omega=0$ as $\dim\omega=\dim\Mman$.

\subsection{Poisson multiplication}\label{sect:Poisson_mult}
Let $\Qman$ be another one-dimensional manifold without boundary, so $\Qman$ is either $\RRR$ or $S^1$ as well as $\Pman$.  
Then $\omega$ yields a \myemph{Poisson} multiplication
\begin{equation}\label{equ:poisson_mult_map}
\{\cdot, \cdot\}: C^{\infty}(\Mman,\Pman) \times C^{\infty}(\Mman,\Qman) \longrightarrow C^{\infty}(\Mman,\mathbb{R})
\end{equation}
defined by one of the following equivalent formulas:
\begin{equation}\label{equ:poisson_mult_formula}
\{\func, \gfunc\} := 
\omega(\hFld_{\func}, \hFld_{\gfunc})  =
\psi(\hFld_{\func})(\hFld_{\gfunc}) =
\df{\hFld_{\func}}{\gfunc} = -\df{\hFld_{\gfunc}}{\func},
\end{equation}
where $\hFld_{\func}$ and $\hFld_{\gfunc}$ are Hamiltonian vector fields of $\func\in C^{\infty}(\Mman,\Pman)$ and $\gfunc\in C^{\infty}(\Mman,\Qman)$ respectively.

In particular, for each $\func\in C^{\infty}(\Mman,\Pman)$ one can define its \myemph{annulator} with respect to~\eqref{equ:poisson_mult_formula} by
\begin{equation}\label{equ:f_annulator}
\ZentrQ{\Qman}{\func} = 
\{
\gfunc \in C^{\infty}(\Mman,\Qman) \mid \df{\hFld_{\func}}{\gfunc} = \{\func,\gfunc\} = 0
\}.
\end{equation}
Thus $\ZentrQ{\Qman}{\func}$ consists of all maps $\gfunc\in C^{\infty}(\Mman,\Qman)$ taking constant values along orbits of the Hamiltonian vector field $\hFld_{\func}$.
It follows from~\eqref{equ:poisson_mult_formula} that $\gfunc\in \ZentrQ{\Qman}{\func}$ iff $\func\in\ZentrQ{\Pman}{\gfunc}$.

When $\Pman=\Qman=\RRR$, this multiplication is the usual \myemph{Poisson bracket}, and $\ZentrQ{\RRR}{\func}$ is the \myemph{centralizer} of $\func$, see~\cite[\S3]{McDuffSalamon:SympTop:1995}.

\subsection{Class $\FF(\Mman,\Pman)$}\label{sect:class_F}
Let $\FF(\Mman,\Pman)$ be the subspace of $C^{\infty}(\Mman,\Pman)$ consisting of maps $\func$ satisfying the following two axioms:
\begin{axiom}{\AxBd}
The map $\func$ takes a constant value at each connected component of $\partial\Mman$ and has no critical points on $\partial\Mman$.
\end{axiom}
\begin{axiom}{\AxCrPt}
For every critical point $z$ of $\func$ there is a local presentation $\hat\func:\RRR^2\to\RRR$ of $\func$ near $z$ in which $\hat\func$ is a homogeneous polynomial $\RRR^2\to\RRR$ without multiple factors.
\end{axiom}
In particular, since the polynomial $\pm x^2 \pm y^2$ (a non-degenerate singularity) is homogeneous and has no multiple factors, we see that $\FF(\Mman,\Pman)$ contains an open and everywhere dense subset $\Morse(\Mman,\Pman)$ consisting of maps satisfying Axiom~\AxBd\ and having non-degenerate critical points only.

Figure~\ref{fig:isol_crit_pt} describes possible singularities satisfying Axiom~\AxCrPt.

\begin{figure}
\begin{tabular}{ccccccccc}
\includegraphics[height=2cm]{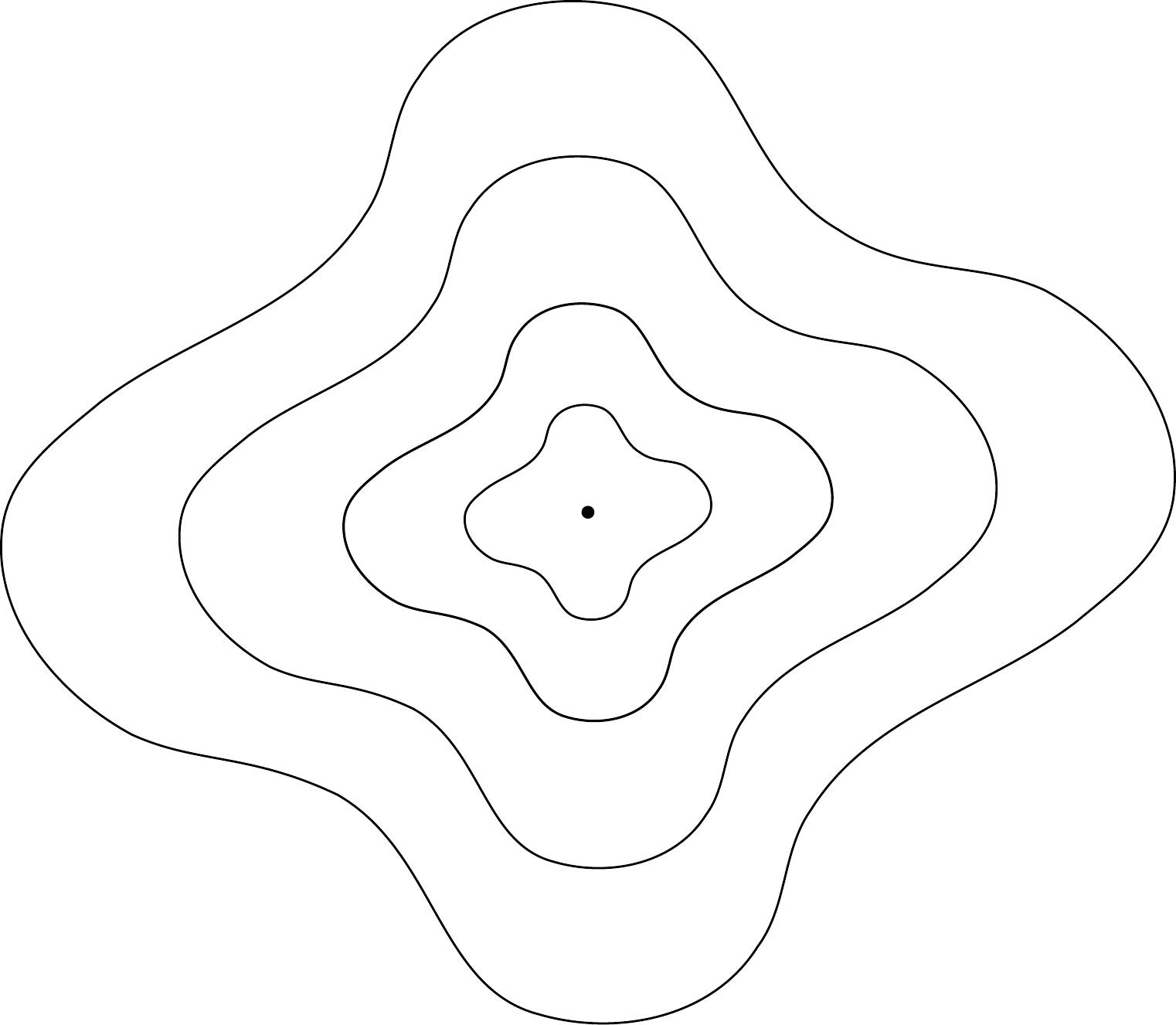} & \quad
\includegraphics[height=2cm]{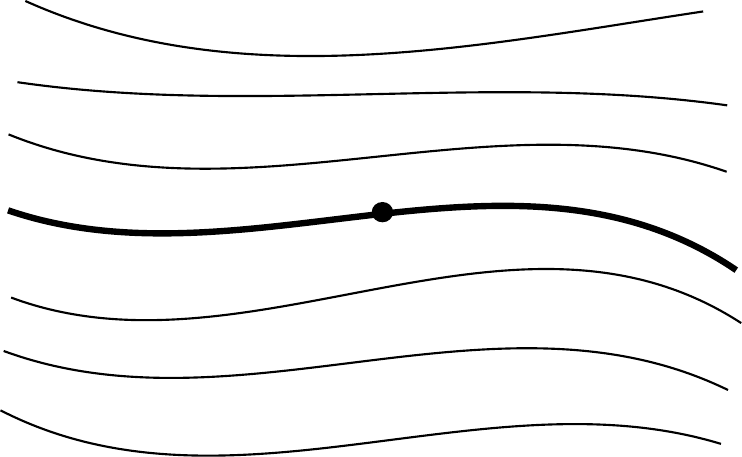} & \quad
\includegraphics[height=2cm]{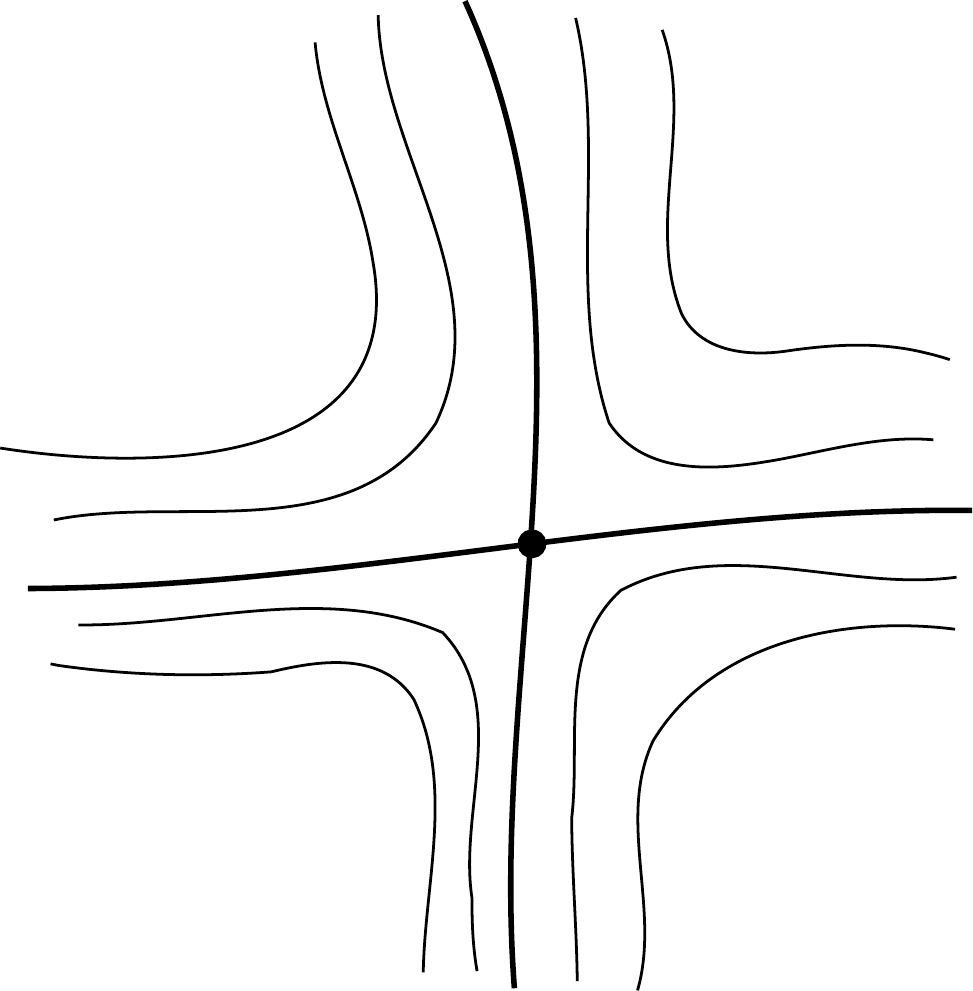} & \quad
\includegraphics[height=2cm]{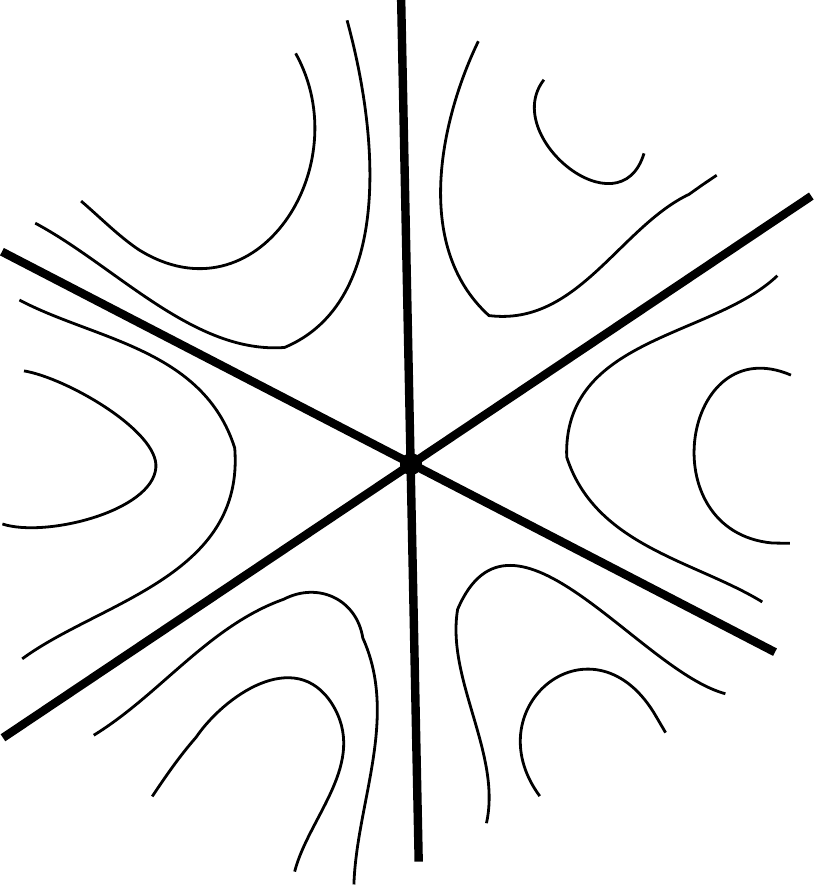} 
\end{tabular}
\caption{Level sets foliation for singularities of Axiom~\AxCrPt}
\label{fig:isol_crit_pt}
\end{figure}

\begin{definition}\label{defn:ham_like_vf}
We will say that a vector field $\fld$ on $\Mman$ is \myemph{Hamiltonian like} for $\func\in\FF(\Mman,\Pman)$ if 
\begin{enumerate}[leftmargin=*, label=$(\alph*)$]
\item\label{enum:dfF_0}
$\df{\fld}{\func}=0$, and, in particular, $\fld$ is tangent to $\partial\Mman$ and generates a flow on $\Mman$;
\item\label{enum:Fz_0__dfz_0}
$\fld(z)=0$ if and only if $z$ is a critical point of $\func$;
\item\label{enum:ham_like}
for each $z$ critical point of $\func$ there exists a local representation $\hat\func:\RRR^2\to\RRR$ of $\func$ as a homogeneous polynomial without multiple factors such that in these coordinates $\fld(x,y) = -\hat\func'_y \tfrac{\partial}{\partial x} + \hat\func'_x \tfrac{\partial}{\partial y}$.
\end{enumerate}
\end{definition}
One can easily prove that for each $\func\in\FF(\Mman,\Pman)$ there exists a Hamiltonian like vector field, \cite[Lemma~5.1]{Maksymenko:AGAG:2006}.

Notice also that every Hamiltonian vector field $\hFld$ of $\func$ has properties~\ref{enum:dfF_0} and~\ref{enum:Fz_0__dfz_0} of Definition~\ref{defn:ham_like_vf}.
Moreover, if $\hFld$ is also a Hamiltonian like, then due to~\eqref{equ:ham_v_f_formula} in the corresponding coordinates satisfying property \ref{enum:ham_like} of Definition~\ref{defn:ham_like_vf} we have that $\omega = dx\wedge dy$.

\begin{lemma}\label{lm:H_lF}
Let $\fld$ be any Hamiltonian like vector field for $\func\in\FF(\Mman,\Pman)$, and $\hFld$ be the Hamiltonian vector field for $\func$ with respect to $\omega$.
Then there exists an everywhere non-zero $C^{\infty}$ function $\lambda:\Mman\to\RRR\setminus\{0\}$ such that $\hFld = \lambda \fld$.
\end{lemma}
\begin{proof}
Denote by $\fSing$ the set of critical point of $\func$, being also the set of zeros of $\hFld$ as well as of $\fld$.
Since $\fld$ and $\hFld$ are parallel and non-zero on $\Mman\setminus\fSing$, it follows that there exists a $C^{\infty}$ non-zero function $\lambda:\Mman\setminus\fSing\to\RRR$ such that $\hFld = \lambda\fld$.
It remains to show that $\lambda$ can be defined by non-zero values on $\fSing$ to give a $C^{\infty}$ function on all of $\Mman$.

Let $z$ be a critical point of $\func$.
Then by definition of Hamiltonian like vector field there exists a local representation $\hat{\func}:\RRR^2\to \RRR$ of $\func$ such that $z=(0,0)\in\RRR^2$, $\hat{\func}$ is a homogeneous polynomial without multiple factors, and $\fld = - \hat{\func}'_{y} \tfrac{\partial}{\partial x} + \hat{\func}'_{x} \tfrac{\partial}{\partial y}$.

Then $\omega(x,y) = \gamma(x,y) dx \wedge dy$ for some non-zero $C^{\infty}$ function $\gamma$, and by formula~\eqref{equ:ham_v_f_formula}, we have $\hFld = \tfrac{1}{\gamma} \fld$ on $\Uman\setminus z$.
Hence $\lambda = 1/\gamma$, and so $\lambda$ smoothly extends to all of $\Uman$ by $\lambda(z) = 1/\gamma(z)$.
\end{proof}

The following statement is a particular case of results of \cite{Maksymenko:CEJM:2009} on parameter rigidity.
\begin{corollary}{c.f. \rm \cite[\S4 \& Theorem~11.1]{Maksymenko:CEJM:2009}}
For any two Hamiltonian like vector fields $\fld_1$ and $\fld_2$ there exists an everywhere non-zero $C^{\infty}$ function $\mu:\Mman\to\RRR\setminus\{0\}$ such that $\fld_1 = \mu\fld_2$.
\end{corollary}
\begin{proof}
It follows from Lemma~\ref{lm:H_lF} that $\hFld=\lambda_1\fld_1=\lambda_2\fld_2$ for some everywhere non-zero $C^{\infty}$ functions $\lambda_1,\lambda_2:\Mman\to\RRR\setminus\{0\}$.
Hence $\mu = \lambda_2 / \lambda_1$.
\end{proof}

\subsection{Topological type of $\StabilizerId{\func}$}
Let $\func\in\FF(\Mman,\Pman)$, $\hFld$ be a Hamiltonian like vector field for $\func$, and $\hFlow:\Mman\times\RRR\to\Mman$ be the corresponding Hamiltonian flow.
\begin{theorem}\label{th:charact_Sid}{\rm\cite{Maksymenko:reparam-sh-map}, \cite{Maksymenko:ProcIM:ENG:2010}, \cite{Maksymenko:OsakaJM:2011}}.
Let $\sh:C^{\infty}(\Mman,\RRR) \to C^{\infty}(\Mman,\Mman)$ be the shift map along orbits of $\hFld$ and
\[
\Gamma = \{\alpha\in C^{\infty}(\Mman,\RRR) \mid 1 + \df{\hFld}{\alpha}>0 \},
\]
see~\eqref{equ:Gamma_set}.
\begin{enumerate}[leftmargin=*, label={\rm(\arabic*)}]
\item\label{enum:th:charact_Sid:sh_G_in_Stabf}
$\sh(\Gamma) = \StabilizerId{\func}$ \ and \ $\Gamma = \sh^{-1}\bigl( \StabilizerId{\func} \bigr)$.
\item\label{enum:th:charact_Sid:s1}
Suppose all critical points of $\func$ are \myemph{non-degenerate local extremes}, so, in particular, $\func\in\Morse(\Mman,\Pman)$. 
Then the restriction map $\sh|_{\Gamma}:\Gamma \to \StabilizerId{\func}$ is an infinite cyclic covering, and so $\StabilizerId{\func}$ is homotopy equivalent to the circle.
More precisely, in this case there exists $\theta\in\Gamma$ such that 
\begin{enumerate}[leftmargin=7ex, label={\rm(\roman*)}]
\item\label{enum:th:charact_Sid:s1:theta_pos}
$\theta>0$ on all of $\Mman$;
\item\label{enum:th:charact_Sid:s1:theta_per}
each non-constant orbit $\gamma$ of $\fld$ is periodic, and $\theta$ takes a constant value on $\gamma$ being an positive integral multiple of the period $\Per(\gamma)$ of $\gamma$;
\item\label{enum:th:charact_Sid:s1:Z_action}
there exists a free action of $\ZZZ$ on $\Gamma$ defined by $n*\alpha = \alpha + n\theta$, for $n\in\ZZZ$ and $\alpha\in\Gamma$, such that 
the map $\sh$ is a composite
\begin{equation}\label{equ:shH_decomp}
\sh: \Gamma \xrightarrow{~~~~~~p~~~~~~} \Gamma/\ZZZ \xrightarrow[\cong]{~~~~~~~r~~~~~~~} \StabilizerId{\func},
\end{equation}
where $p$ is a projection onto the factor space $\Gamma/\ZZZ$ endowed with the corresponding final topology, and $r$ is a homeomorphism.
\end{enumerate}
\item\label{enum:th:charact_Sid:pt}
Suppose $\func$ has a critical point being \myemph{not a non-degenerate local extreme}.
Then $\sh|_{\Gamma}:\Gamma \to \StabilizerId{\func}$ is a homeomorphism, and so $\StabilizerId{\func}$ is contractible.
\end{enumerate}
\end{theorem}
\begin{proof}
In fact, Theorem~\ref{th:charact_Sid} is stated and proved in~\cite{Maksymenko:ProcIM:ENG:2010} for any \myemph{Hamiltonian like} vector field $\fld$ of $\func$.
The advantage of using Hamiltonian like vector fields is that we have \myemph{precise} formulas for $\fld$ near critical points of $\func$.

Let $\lambda:\Mman\to\RRR$ be every where non-zero $C^{\infty}$ function and $\hFld = \lambda \fld$.
We will deduce from results of~\cite{Maksymenko:reparam-sh-map} that Theorem~\ref{th:charact_Sid} also holds for $\hFld$.
Due to Lemma~\ref{lm:H_lF} this includes the case when $\hFld$ is Hamiltonian.

Let $\flow, \hFlow:\Mman\times\RRR\to\Mman$ be the flows of $\fld$ and $\hFld=\lambda\fld$ respectively, 
\[\sh_{\fld},\sh_{\hFld}:C^{\infty}(\Mman,\RRR) \to C^{\infty}(\Mman,\Mman)\] be the corresponding shift maps, 
$\image(\sh_{\hFld})$, $\image(\sh_{\fld})$ be their images in $C^{\infty}(\Mman,\Mman)$, and $\Gamma_{\fld}$, $\Gamma_{\hFld}$ be corresponding the subsets of $C^{\infty}(\Mman,\RRR)$ defined by~\eqref{equ:Gamma_set}.
Define the following $C^{\infty}$ function
\begin{align*}
&\sigma:\Mman\times\RRR\to\RRR, &
\sigma(x,s) = &
\int_{0}^{s}\lambda(\hFld(x,t))dt.
\end{align*}
Then it is well known and easy to see, e.g.~\cite{Maksymenko:reparam-sh-map}, that for each $\alpha\in C^{\infty}(\Mman,\RRR)$ we have that
\begin{equation}\label{equ:Ha_Fga}
\hFlow(x,\alpha(x)) = \flow(x,\sigma(x,\alpha(x))).
\end{equation}
Consider the map
\begin{align*}
&\gamma:C^{\infty}(\Mman,\RRR) \to C^{\infty}(\Mman,\RRR), &
\gamma(\alpha)(x) &= \sigma(x,\alpha(x)).
\end{align*}
Evidently, $\gamma$ is continuous with respect to $C^{\infty}$ topologies.
Moreover, \eqref{equ:Ha_Fga} means that $\hFlow_{\alpha} = \flow_{\gamma(\alpha)}$ for all $\alpha\in C^{\infty}(\Mman,\RRR)$.
Hence $\sh_{\hFld} = \sh_{\fld}\circ\gamma$, $\image(\sh_{\hFld}) \subset \image(\sh_{\fld})$, and we get the following commutative diagram:
\[
\xymatrix{
\Gamma_{\hFld} \ar@{^(->}[r] & C^{\infty}(\Mman,\RRR) \ar[r]^{\sh_{\hFld}} \ar[d]_{\gamma} & \image(\sh_{\hFld}) \ar@{^(->}[d] \ar@{^(->}[r] & C^{\infty}(\Mman,\Mman) \\
\Gamma_{\fld} \ar@{^(->}[r] & C^{\infty}(\Mman,\RRR) \ar[r]^{\sh_{\fld}} & \image(\sh_{\fld}) \ar@{^(->}[r] & C^{\infty}(\Mman,\Mman) \\
}
\]
Since $\lambda\not=0$ everywhere, one can interchange $\fld = \tfrac{1}{\lambda}\hFld$ and $\hFld$.
Hence by the same arguments as above we get that $\image(\sh_{\hFld}) = \image(\sh_{\fld})$ and $\gamma$ is a homeomorphism.
Also notice that the orbit structures of $\fld$ and $\hFld$ coincide.
Hence $\DiffId(\fld) = \DiffId(\hFld)$, and so
\begin{align*}
\Gamma_{\hFld} 
&\stackrel{\eqref{equ:shift_being_diffeos}}{=\!=\!=} \sh_{\hFld}^{-1} \bigl( \image(\sh_{\hFld}) \cap \DiffId(\hFld) \bigr) 
= \sh_{\hFld}^{-1} \bigl( \image(\sh_{\fld}) \cap \DiffId(\fld) \bigr) \\
&= \gamma^{-1} \circ \sh_{\fld}^{-1} \bigl( \image(\sh_{\fld}) \cap \DiffId(\fld) \bigr) 
\stackrel{\eqref{equ:shift_being_diffeos}}{=\!=\!=} \gamma^{-1}(\Gamma_{\fld}).
\end{align*}
Thus $\gamma$ yields a homeomorphism of $\Gamma_{\hFld}$ onto $\Gamma_{\fld}$.
Since Theorem~\ref{th:charact_Sid} holds for $\fld$, we get the following commutative diagram
\[
\xymatrix{
\Gamma_{\hFld} \ar[rr]^-{\gamma}_-{\cong} \ar[rd]_-{\sh_{\hFld}} && \Gamma_{\fld} \ar[ld]^-{\sh_{\fld}} \\
& \StabilizerId{\func} 
}
\]
implying that $\sh_{\hFld}|_{\Gamma_{\hFld}}$ has the same topological properties~\ref{enum:th:charact_Sid:sh_G_in_Stabf}-\ref{enum:th:charact_Sid:pt} as $\sh_{\fld}|_{\Gamma_{\fld}}$, and so Theorem~\ref{th:charact_Sid} holds for $\hFld$ as well.
\end{proof}

\begin{remark}\rm
Let us discus the case~\ref{enum:th:charact_Sid:s1} of Theorem~\ref{th:charact_Sid} which is realized precisely for the following four types of Morse maps, see~\cite[Theorem~1.9]{Maksymenko:AGAG:2006}: 
\begin{enumerate}[leftmargin=*, label=(\Alph*)]
\item
$\Mman=S^2$ is a $2$-sphere and $f:S^2\to\Pman$ has exactly two critical points: non-degenerate local minimum and maximum;
\item
$\Mman=D^2$ is a $2$-disk and $f:D^2\to\Pman$ has exactly one critical point being a non-degenerate local extreme;
\item
$\Mman=S^1\times[0,1]$ is a cylinder and $f:S^1\times[0,1]\to\Pman$ has no critical points;
\item
$\Mman=T^2$ is a $2$-torus, $\Pman=S^1$ is a circle, and $f:T^2\to \Pman$ has no critical points.
\end{enumerate}
Due to~\ref{enum:th:charact_Sid:s1:theta_pos} and~\ref{enum:th:charact_Sid:s1:theta_per} each regular point $x\in\Mman$ of $\func$ is periodic of some period $\Per(x)$, and there exists $k_x\in\NNN$ depending on $x$ such that $\theta(x) = k_x\Per(x)$.
Hence 
\[
\hFlow_{\theta}(x) = \hFlow(x,\theta(x))
= \hFlow(x,k_x \mathrm{Per}(\gamma)) = x,
\]
and so $\hFlow_{\theta} = \id_{\Mman}$.
Moreover, if $\alpha\in\Gamma$, then
\begin{align*}
\hFlow_{\alpha + n\theta}(x) &= \hFlow(x,\alpha(x) + n\theta(x)) = 
\hFlow\bigl( \hFlow(x,n\theta(x)), \alpha(x)\bigr) = \\
&= \hFlow\bigl(x, \alpha(x)\bigr) = \hFlow_{\alpha}(x).
\end{align*}
This implies correctness of the $\ZZZ$-action from~\ref{enum:th:charact_Sid:s1:Z_action} of Theorem~\ref{th:charact_Sid} and existence of decomposition~\eqref{equ:shH_decomp} with continuous $p$ and $r$.
The principal difficulty was to prove that $r$ is a homeomorphism.
\end{remark}

The aim of the present paper is to deduce from Theorem~\ref{th:charact_Sid} a description of the homotopy type of  $\StabilizerId{\func,\omega}$, see Theorem~\ref{th:charact_Sid_f_omega} below.

\section{Main result}
Let $\Mman$ be a compact orientable surface equipped with a volume form $\omega$, $\func\in\FF(\Mman,\Pman)$, $\hFld$ be the Hamiltonian vector field of $\func$ with respect to $\omega$, $\hFlow:\Mman\times\RRR\to\Mman$ be the corresponding Hamiltonian flow, and 
\begin{align*}
&\sh:C^{\infty}(\Mman,\RRR) \to C^{\infty}(\Mman,\Mman), &
\sh(\alpha)(x) = \hFlow(x,\alpha(x))
\end{align*}
be the shift map along orbits of $\hFlow$.
Let also
\[
\mathcal{Z} = \ZentrQ{\RRR}{\func} = \{\alpha\in C^{\infty}(\Mman,\RRR) \mid \df{\hFld}{\alpha}=0 \}
\]
be the space of functions taking constant values along orbits of $\hFld$, see~\eqref{equ:f_annulator}.
Then $\mathcal{Z}$ is a linear subspace of $C^{\infty}(\Mman,\RRR)$ and is contained in $\Gamma$, see~\eqref{equ:Gamma_set}.
In particular, $\mathcal{Z}$ is contractible as well as $\Gamma$.

\begin{theorem}\label{th:charact_Sid_f_omega}
The following statements hold true.
\begin{enumerate}[leftmargin=*, label=$(\arabic*)$]
\item\label{enum:th:charact_Sid_f_omega:sh_Z_in_Stabfo}
$\sh(\mathcal{Z}) = \StabilizerId{\func,\omega} = \StabilizerId{\func} \cap \Stabilizer{\func,\omega}$ and
$\mathcal{Z} = \sh^{-1}\bigl(\StabilizerId{\func,\omega}\bigr)$.
\item\label{enum:th:charact_Sid_f_omega:s1}
If all critical points of $\func$ are non-degenerate local extremes, then the restriction $\sh|_{\mathcal{Z}}:\mathcal{Z} \to\StabilizerId{\func,\omega}$ is an infinite cyclic covering, and $\StabilizerId{\func,\omega}$ is homotopy equivalent to the circle.
\item\label{enum:th:charact_Sid_f_omega:pt}
Otherwise, $\sh|_{\mathcal{Z}}:\mathcal{Z} \to\StabilizerId{\func,\omega}$ is a homeomorphism, and so $\StabilizerId{\func,\omega}$ is contractible.
\item\label{enum:th:charact_Sid_f_omega:hom_eq}
The inclusion $\StabilizerId{\func,\omega} \subset \StabilizerId{\func}$ is a homotopy equivalence.
\item\label{enum:th:charact_Sid_f_omega:j_is_inj}
The inclusion map $j:\Stabilizer{\func,\omega} \subset \Stabilizer{\func}$ induces an injection $j_0:\pi_0\Stabilizer{\func,\omega}\to \pi_0\Stabilizer{\func}$.
\end{enumerate}
\end{theorem}
\begin{proof}
First we need the following lemma.
\begin{lemma}
Let $\alpha\in\Cinfty(\Mman,\RRR)$.
Then the action of $\hFlow_{\alpha}$ on $\omega$ is given by
\begin{align}\label{equ:flow_alpha_acts_on_omega}
\hFlow_{\alpha}^{*} \omega &= (1+\df{\hFld}{\alpha}) \cdot \omega.
\end{align}
\end{lemma}
\begin{proof}
Since the set of critical points is finite and so nowhere dense, it suffices to check this relation at regular points of $\func$ only.

So let $p$ be a regular point of $\func$.
Then $\hFld(p) \not=0$, whence there are local coordinates $(x,y)$ at $p$ in which $p=(0,0)$, $\hFld(x,y) = \tfrac{\partial}{\partial x}$, and $\hFlow(x,y,t) = (x+t,y)$ for sufficiently small $x,y,t$.
In particular, $\df{\hFld}{\alpha} = \tfrac{\partial\alpha}{\partial x}$.
We also have that $\omega(x,y) = \gamma(x,y) dx \wedge dy$ for some $C^{\infty}$ function $\gamma$.

Notice that one may also assume that $\alpha(p)=0$.
Indeed, let $b = \alpha(p)$.
Then $\hFlow_{\alpha} = \hFlow_{\alpha-b} \circ \hFlow_b$.
Since $\hFlow_b$ preserves $\omega$, see~\eqref{equ:Ht_prev_omega}, it follows that
\[
\hFlow_{\alpha}^{*}\omega = \hFlow_{\alpha-b}^{*} \circ \hFlow_b^{*} \omega = \hFlow_{\alpha-b}^{*} \omega.
\]
Thus suppose $\alpha(p)=0$, whence $\hFlow_{\alpha}(p) = p$.
Then
\begin{align*}
\hFlow_{\alpha}^{*} \omega(x,y) 
&= \gamma\circ\hFlow_{\alpha}(x,y) \, d(x+\alpha) \wedge dy \\
&= \gamma(x+\alpha, y) \, (1+\alpha'_{x}) \, dx \wedge dy \\
&= \gamma(x+\alpha, y) \, \bigl(1+\df{\hFld}{\alpha}\bigr) \, dx \wedge dy.
\end{align*}
In particular, at $p$ we have that 
\begin{align*}
\hFlow_{\alpha}^{*} \omega(p) = (1+\df{\hFld}{\alpha}(p)) \cdot \omega(p),
\end{align*}
which proves~\eqref{equ:flow_alpha_acts_on_omega}.
\end{proof}

Now we can complete Theorem~\ref{th:charact_Sid_f_omega}.

\ref{enum:th:charact_Sid_f_omega:sh_Z_in_Stabfo}
Let us check that 
\begin{equation}\label{equ:shZ_S0fo}
\sh(\mathcal{Z}) \subset \StabilizerId{\func,\omega}.
\end{equation}
Let $\alpha\in\mathcal{Z}$.
As $\hFlow_{\alpha}$ leaves invariant each orbit of $\hFld$, and therefore it preserves $\func$, we have that $\hFlow_{\alpha}\in\Stabilizer{\func}$.

Moreover, by formula~\eqref{equ:flow_alpha_acts_on_omega}, $\hFlow_{\alpha}^{*}\omega=\omega$, so $\hFlow_{\alpha}\in\Stabilizer{\func,\omega}$.

Now notice that $t\alpha \in\mathcal{Z}$ for all $t\in\RRR$, and so $\hFlow_{t\alpha}\in\Stabilizer{\func,\omega}$ as well.
Thus the homotopy $\hFlow_{t\alpha}:\Mman\to\Mman$, $t\in[0,1]$, is in fact an isotopy in $\Stabilizer{\func,\omega}$ between $\id_{\Mman} = \hFlow_0$ and $\hFlow_{\alpha}$.
Hence $\hFlow_{\alpha}\in\StabilizerId{\func,\omega}$.

\smallskip

Further we claim that
\begin{equation}\label{equ:Z_sh1_S0f_Sfo}
\mathcal{Z} \ \supset \ \sh^{-1}\bigl( \StabilizerId{\func} \cap \Stabilizer{\func,\omega} \bigr).
\end{equation}
Indeed, let $\dif\in\StabilizerId{\func} \cap \Stabilizer{\func,\omega}$.
Then by~\ref{enum:th:charact_Sid:sh_G_in_Stabf} of Theorem~\ref{th:charact_Sid} $\dif = \hFlow_{\alpha}$ for some $\alpha\in\Gamma$.
As $\omega$ is everywhere non-zero on $\Mman$, it follows from formula~\eqref{equ:flow_alpha_acts_on_omega} that $\df{\hFld}{\alpha}=0$ on all of $\Mman$, that is $\alpha\in\mathcal{Z}$.

Hence 
\[
\sh(\mathcal{Z}) 
\ \stackrel{\eqref{equ:shZ_S0fo}}{\subset} \ 
\StabilizerId{\func,\omega} \ \ \subset \ \ \StabilizerId{\func} \cap \Stabilizer{\func,\omega} \
\ \stackrel{\eqref{equ:Z_sh1_S0f_Sfo}}{\subset} \
\sh(\mathcal{Z}),
\]
\[
\mathcal{Z} 
\ \stackrel{\eqref{equ:shZ_S0fo}}{\subset} \ 
\sh^{-1}\bigl( \StabilizerId{\func,\omega} \bigr) \ \subset \ 
\sh^{-1}\bigl( \StabilizerId{\func} \cap \Stabilizer{\func,\omega} \bigr) 
\ \stackrel{\eqref{equ:Z_sh1_S0f_Sfo}}{\subset} \
\mathcal{Z}.
\]
This proves~\ref{enum:th:charact_Sid_f_omega:sh_Z_in_Stabfo}.

\medskip

\ref{enum:th:charact_Sid_f_omega:s1}, \ref{enum:th:charact_Sid_f_omega:hom_eq}
Suppose $\sh:\Gamma\to\StabilizerId{\func}$ is an infinite cyclic covering map, and let $\theta\in\Gamma$ be the function from~\ref{enum:th:charact_Sid:s1} of Theorem~\ref{th:charact_Sid}.

Then due to property~\ref{enum:th:charact_Sid:s1:theta_per} in Theorem~\ref{th:charact_Sid}	 $\theta$ takes constant values along orbits of $\hFld$, and therefore $\theta \in \mathcal{Z}$.
Since, in addition, $\mathcal{Z}$ is a group, it follows that $\mathcal{Z}$ is invariant with respect to the $\ZZZ$-action on $\Gamma$, i.e. $\alpha+n\theta\in\mathcal{Z}$ for all $\alpha\in\mathcal{Z}$.
Therefore $\mathcal{Z} = \sh^{-1}(\StabilizerId{\func,\omega})$.
Hence $\sh|_{\mathcal{Z}}: \mathcal{Z} \to \StabilizerId{\func,\omega}$ is a $\ZZZ$-covering as well as $\sh|_{\Gamma}$.
As $\mathcal{Z}$ is contractible, we obtain that the quotient $\StabilizerId{\func,\omega}$ is homotopy equivalent to the circle.

Consider the following path $\tau:[0,1]\to\mathcal{Z} \subset\Gamma$, $\tau(t) = t \theta$.
Then $\sh\circ\tau$ is a loop in $\StabilizerId{\func,\omega} \subset \StabilizerId{\func}$, since
\[
\sh\circ\tau(1)(x) = \flow(x,\theta(x))= x = \flow(x,0) = \sh\circ\tau(0)(x).
\]
This loop is a generator of $\pi_1\StabilizerId{\func,\omega}\cong\ZZZ$ as well as a generator of $\pi_1\StabilizerId{\func}\cong\ZZZ$.
Hence the inclusion $j:\StabilizerId{\func,\omega} \subset \StabilizerId{\func}$ yields an isomorphism of fundamental groups.
Since these spaces homotopy equivalent to the circle, we obtain that $j$ is a homotopy equivalence.

\medskip

\ref{enum:th:charact_Sid_f_omega:pt}, \ref{enum:th:charact_Sid_f_omega:hom_eq}
If $\sh:\Gamma\to\StabilizerId{\func}$ is a homeomorphism, then due to \ref{enum:th:charact_Sid_f_omega:sh_Z_in_Stabfo} it yields a homeomorphism of $\mathcal{Z}$ onto $\StabilizerId{\func,\omega}$.
In particular, both $\StabilizerId{\func,\omega}$ and $\StabilizerId{\func}$ are contractible, and so the inclusion $\StabilizerId{\func,\omega} \subset \StabilizerId{\func}$ is a homotopy equivalence.

\medskip

\ref{enum:th:charact_Sid_f_omega:j_is_inj}
Injectivity of $j_0$ follows from the relation $\StabilizerId{\func,\omega} = \StabilizerId{\func} \cap \Stabilizer{\func,\omega}$.
Theorem~\ref{th:charact_Sid_f_omega} is completed.
\end{proof}

\begin{remark}\rm
Though the inclusion $\StabilizerId{\func,\omega} \subset \StabilizerId{\func}$ is a homotopy equivalence, it seems to be a highly non-trivial task to find precise formulas for a \myemph{strong} deformation retraction of $\StabilizerId{\func}$ onto $\StabilizerId{\func,\omega}$.
For the case~\ref{enum:th:charact_Sid_f_omega:pt} of Theorem~\ref{th:charact_Sid_f_omega} this is equivalent to a construction of a strong deformation retraction of $\Gamma$ onto $\mathcal{Z}$.
In fact, it suffices to find a retraction $r:\Gamma\to\mathcal{Z}$, so to associate to each $\alpha\in\Gamma$ a function $r(\alpha)$ taking constant values along orbits of $\hFld$ so that each $\beta\in\mathcal{Z}$ remains unchanged.
Then a strong deformation $r_t:\Gamma\to\mathcal{Z}$, $t\in[0,1]$, of $\Gamma$ onto $\mathcal{Z}$ can be given by $r_t(\alpha) = (1-t)\alpha+tr(\alpha)$.
\end{remark}

\subsection{Counterexample for maps $\gfunc\not\in\FF(\Mman,\Pman)$}\label{counterexample:Zg_Sg}
Let $D^2 = \{|z|\leq 1\}$ be the unit disk in the complex plane $\CCC$ and $\omega = dx\wedge dy$ be the standard symplectic form.
Consider the following two functions $\func, \gfunc: D^2\to [0,1]$ defined by 
\begin{align*}
\func(x,y) &= x^2+y^2 = |z|^2, &
\gfunc(x,y) &= (x^2+y^2)^2 = |z|^4.
\end{align*}
Then the foliations by level sets of $\func$ and $\gfunc$ coincide, whence 
\begin{align*}
\ZentrQ{\RRR}{\func} &= \ZentrQ{\RRR}{\gfunc}, & \ \ 
\StabilizerId{\func,\omega} &= \StabilizerId{\gfunc,\omega}, &  \ \
\StabilizerId{\func} &= \StabilizerId{\gfunc}.
\end{align*}
However, $\func\in\FF(D^2,\RRR)$, while $\gfunc$ does not belong to $\FF(D^2,\RRR)$ since it is a polynomial with multiple factors.

Notice that the Hamiltonian vector fields $F$ and $G$ of $\func$ and $\gfunc$ are given by
\begin{align*}
F(x,y) &= -2y\tfrac{\partial}{\partial x} + 2x\tfrac{\partial}{\partial y}, &
G(x,y) &= 2(x^2+y^2) F(x,y).
\end{align*}
In particular, the Hamiltonian flow $\mathbf{F}: D^2\times\RRR \to D^2$ of $\func$ is given by $\mathbf{F}(z,t) = e^{2i t} z$, and so the tangent map $T_0 \mathbf{F}_t: T_0 D^2\to T_0 D^2$ is not the identity for $t\not=\pi n$, $n\in\ZZZ$.

On the other hand, the linear part of $G$ at $0$ vanishes, whence for the Hamiltonian flow $\mathbf{G}: D^2\times\RRR \to D^2$ of $\gfunc$ the corresponding tangent map
$T_0 \mathbf{G}_t: T_0 D^2\to T_0 D^2$ is always the identity.
Hence for any $C^{\infty}$ function $\alpha$ the tangent map at $0$ of $\mathbf{G}_{\alpha}$ is the identity as well.
Therefore for $t\not=\pi n$, $n\in\ZZZ$, then $\mathbf{F}_t \not= \mathbf{G}_{\alpha}$ for any $\alpha\in C^{\infty}(\Mman,\RRR)$.

By Theorem~\ref{th:charact_Sid_f_omega} the shift map $\sh_{F}:\ZentrQ{\RRR}{\func} \to \StabilizerId{\func,\omega}$ of $F$ is an infinite cyclic covering, while the shift map 
\[
\sh_{G}:\ZentrQ{\RRR}{\gfunc} \equiv \ZentrQ{\RRR}{\func} \ \longrightarrow \ \StabilizerId{\func,\omega} \equiv \StabilizerId{\gfunc,\omega}
\]
of $G$ turns out to be not surjective, since its image does not contain $\mathbf{F}_t$ for $t\not=\pi n$, $n\in\ZZZ$.

Thus we see that the centralizer of $\gfunc$ does not ``detect'' all the diffeomorphisms from $\StabilizerId{\gfunc}$, while the centralizer of $\func$ does so.
This shows that the assumption $\func\in\FF(\Mman,\Pman)$ in Theorem~\ref{th:charact_Sid_f_omega} is essential.

\subsection{Kronrod-Reeb graph of $\func$}\label{sect:Kronrod_Reeb_graph}
Now we will give an interpretation of $\mathcal{Z}$ in terms of functions on the Kronrod-Reeb graph of $\func$.

Let $\func\in\FF(\Mman,\Pman)$.
Consider the partition $\Delta$ of $\Mman$ into connected components of level-sets of $\func$.
Let $\KRGraph:=\Mman/\Delta$ be the corresponding quotient space and $p:\Mman\to\KRGraph$ be the factor map.
Then we have a natural decomposition
\[
\func = \widehat{\func} \circ p: \Mman \xrightarrow{~~p~~} \KRGraph \xrightarrow{~~\widehat{f}~~} \Pman.
\]
Endow $\KRGraph$ with the final topology, so a subset $U\subset \KRGraph$ is open if and only if $p^{-1}(U)$ is open in $\Mman$.
Then it is well known that $\KRGraph$ has a natural structure of a one-dimensional CW-complex.
It is called a \myemph{Lyapunov} or \myemph{Kronrod-Reeb} graph of $\func$, \cite{AdelsonWelskyKronrod:DANSSSR:1945}, \cite{Reeb:ASI:1952}, \cite{Kronrod:UMN:1950}, \cite{Franks:Top:1985}, \cite{BolsinovFomenko:1997}.

We will briefly recall the correspondence between elements of $\Delta$ (i.e. points of $\KRGraph$) and orbits of $\hFld$.
Let $\gamma\in\Delta$.
If $\gamma$ contains at least one critical point of $\func$, then it follows from Axiom~\AxCrPt\ that $\gamma$ is a connected 1-dimensional CW-complex such that each of its vertices has even (possibly zero) degree, and $p(\gamma)$ is a vertex of $\KRGraph$.
In this case the vertices of $\gamma$ are critical points of $\func$ being also zeros of $\hFld$, while edges of $\gamma$ are non-closed orbits of $\hFld$.

If $\gamma$ has no critical point of $\func$, then $\gamma$ is a closed orbit of $\hFld$.

\begin{lemma}\label{lm:Z_as_func_on_KR}
Each $\alpha\in\ZentrQ{\Qman}{\func}$ yields a unique continuous function $\widehat{\alpha}:\KRGraph\to\Qman$ such that $\alpha = \widehat{\alpha} \circ p$.
Moreover, the correspondence $\alpha \mapsto \widehat{\alpha}$ is a continuous injective map $\eta:\ZentrQ{\Qman}{\func} \to C(\KRGraph,\Qman)$ with respect to $C^{0}$ topology on $C(\KRGraph,\Qman)$.
\end{lemma}
\begin{proof}
Let $\alpha\in\ZentrQ{\Qman}{\func}$, so $\alpha$ takes constant values along orbits of $\hFld$.
First we should show that $\alpha$ takes constant value at each element of $\Delta$.

Consider any element $\gamma \in \Delta$.
If $\gamma$ contains no critical point of $\func$, then $\gamma$ is a closed orbit of $\hFld$, and so $\alpha$ takes a constant value at $\gamma$, see Figure~\ref{fig:KRGraph}.
\begin{figure}[ht]
\includegraphics[height=3cm]{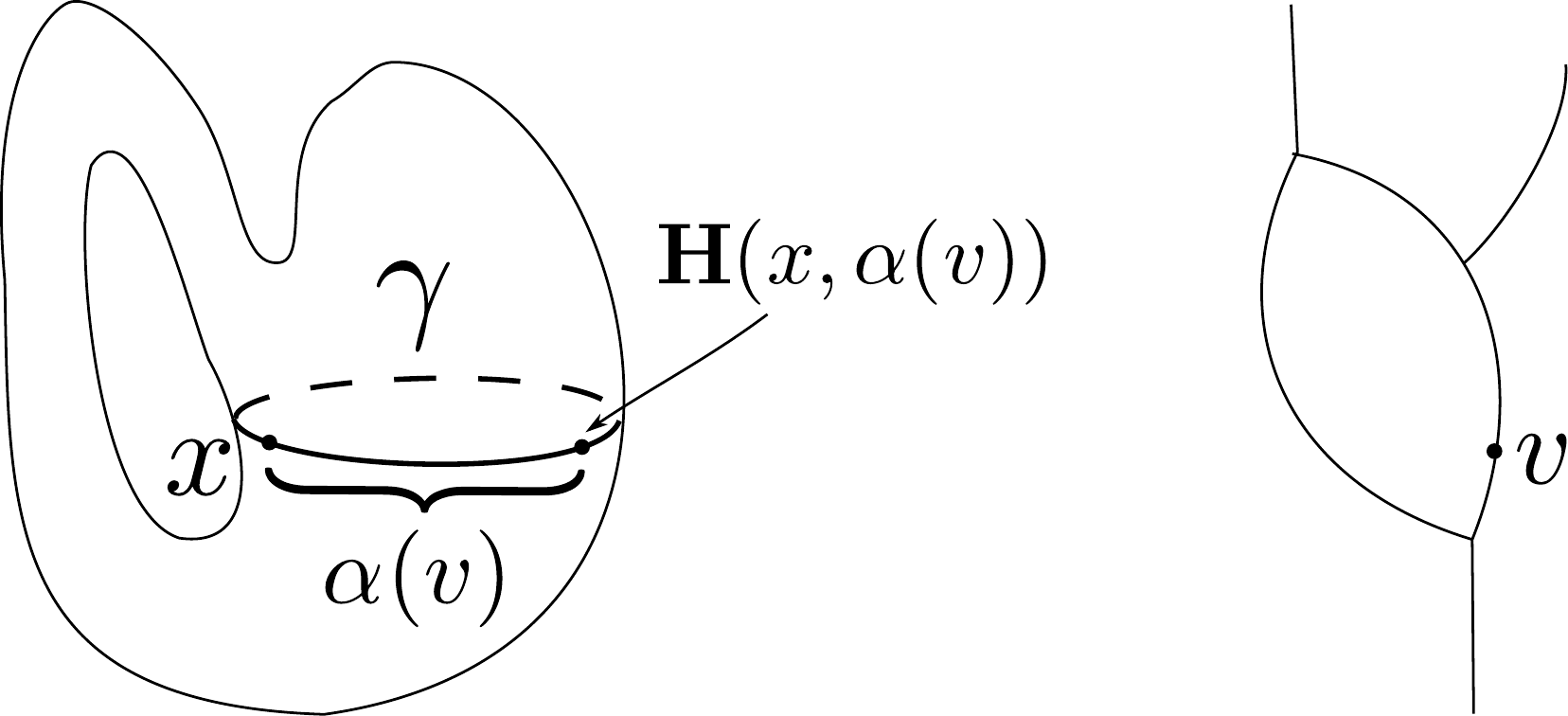}
\caption{}
\label{fig:KRGraph}
\end{figure}

Otherwise, $\gamma$ is a connected $1$-dimensional CW-complex whose vertices and edges are orbits of $\hFld$.
Then $\alpha$ takes constant values along edges of $\gamma$, and it follows from continuity of $\alpha$ and connectedness of $\gamma$ that $\alpha$ is constant on all of $\gamma$.

Thus $\alpha$ yields a unique function $\widehat{\alpha}:\KRGraph\to\Qman$ such that $\alpha = \widehat{\alpha} \circ p$.

Since $\KRGraph$ has final topology with respect to $p$ and $\alpha$ is continuous, it easily follows and is well known that $\widehat{\alpha}$ is continuous.
Continuity of the correspondence $\alpha\mapsto\widehat{\alpha}$ is left for the reader.
\end{proof}

Lemma~\ref{lm:Z_as_func_on_KR} together with Theorem~\ref{th:charact_Sid_f_omega} implies that for $\func\in\FF(\Mman,\Pman)$ the elements of $\StabilizerId{\func,\omega}$ are parametrized by continuous functions on the Kronrod-Reeb graph $\KRGraph$ of $\func$, see Figure~\ref{fig:KRGraph}.

More precisely, due to Theorem~\ref{th:charact_Sid_f_omega} for each $\dif \in \StabilizerId{\func,\omega}$ there exists $\alpha\in\ZentrQ{\Qman}{\func}$ such that $\dif = \hFlow_{\alpha}$.
This function takes constant values on connected components of level-sets of $\func$, and therefore induces a continuous function $\widehat{\alpha}:\KRGraph\to\RRR$.
Then the value of $\widehat{\alpha}$ at some point $v\in\KRGraph$ equals to the common time shift induced by $\dif$ on all the orbits of $\hFld$ constituting $p^{-1}(v)$.

\subsection{Non-surjectivity of the map $j_0:\pi_0\Stabilizer{\func,\omega}\to \pi_0\StabilizerPlus{\func}$}\label{sect:j0_non_surj}
Let $\gfunc:\RRR^2 \to \RRR$ be the function defined by 
\[ 
\gfunc(x,y) = \bigl((x+1)^2 + y^2\bigr) \bigl((x-1)^2 + y^2\bigr).
\]
It has three critical points: one saddle $p_0 = (0,0)$ and two local minimums $p_1=(-1,0)$ and $p_2=(1,0)$.
Let $D = \gfunc^{-1}[0,2]$, and $\func=\gfunc|_{D}:D\to\RRR$ be the restriction of $\gfunc$ to $D$.
Then $D$ is a $2$-disk and $\func$ belongs to the class $\FF(D,\RRR)$.

Consider the following subset $A = \func^{-1}[0,0.5] \subset D$, see Figure~\ref{fig:j0_not_surj}.
It consists of two connected components $A_1$ and $A_2$ containing the points $p_1$ and $p_2$ respectively.
\begin{figure}[ht]
\includegraphics[height=3cm]{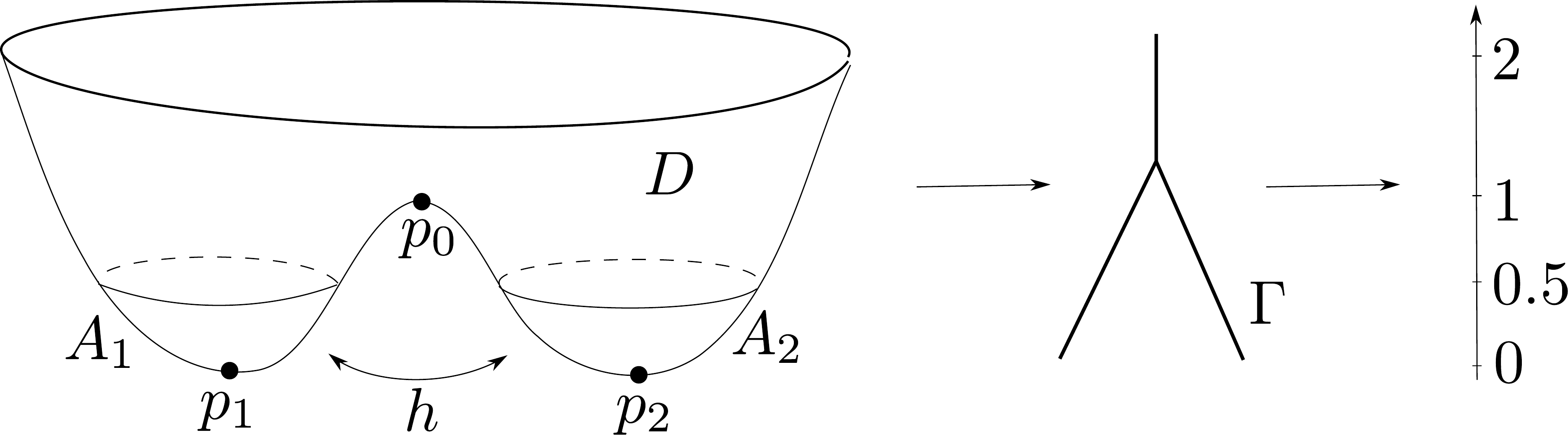}
\caption{}
\label{fig:j0_not_surj}
\end{figure}
Notice that $\dif(A)=A$ for each $\dif\in\Stabilizer{\func}$, whence $\dif$ either preserves both $A_1$ and $A_2$ or interchanges them.
Also notice that if $\dif,k\in\Stabilizer{\func}$ and $\dif(A_1) = A_2$, while $k(A_1)=A_1$, then $\dif$ and $k$ belong to distinct path components of $\Stabilizer{\func}$.

Let $\dif:D\to D$ be a diffeomorphism defined by $D(x,y) = (-x,-y)$.
Evidently, $\dif$ belongs to $\StabilizerPlus{\func}$ and interchanges $A_1$ and $A_2$.

\begin{lemma}
Let $\omega$ be any volume form on $D$ such that $\Vol_{\omega}(A_1) \not= \Vol_{\omega}(A_2)$.
Then the isotopy class $[\dif] \in \pi_0\StabilizerPlus{\func}$ of $\dif$ does not contain any $k\in\Stabilizer{\func,\omega}$.
Hence for such an $\omega$ the map $j_0:\pi_0\Stabilizer{\func,\omega}\to \pi_0\StabilizerPlus{\func}$ is not surjective.
\end{lemma}
\begin{proof}
Each $k\in\Stabilizer{\func,\omega}$ preserves $\omega$-volume.
Since $\Vol_{\omega}(A_1) \not= \Vol_{\omega}(A_2)$, it follows that $k(A_i)=A_i$ for $i=1,2$.
But $\dif(A_1)=A_2$, whence $\dif$ and $k$ are not isotopic in $\Stabilizer{\func}$.
\end{proof}

\def\cprime{$'$}

\end{document}